\documentclass[11pt,reqno]{amsart}
\usepackage{amsmath, amsthm, amssymb}
\usepackage{mathrsfs}
\usepackage{array}
\usepackage{caption}
\evensidemargin \oddsidemargin
\def\bfy{{\mathbf y}}
\def\bfx{{\mathbf x}}

\newtheorem{thm}{Theorem}

\newtheorem{lem}{Lemma}

\newtheorem{prop}{Proposition}
\newcommand{\mmod}[1]{\,\,(\text{\rm mod}\,\, #1)}

\numberwithin{equation}{section} \numberwithin{thm}{section}
\numberwithin{lem}{section} \numberwithin{problem}{section}
\numberwithin{cor}{section}

\def\grm{{\mathfrak m}}\def\grM{{\mathfrak M}}\def\grN{{\mathfrak N}}

\begin{document}
\title[Sums of three cubes]{On squares of sums of three cubes}
\author[Javier Pliego]{Javier Pliego}
\address{Department of Mathematics, Purdue University, 150 N. University Street, West Lafayette, IN 47907-2067, USA}

\email{jp17412@bristol.ac.uk}

\begin{abstract} We show that almost every positive integer can be expressed as a sum of four squares of integers represented as the sums of three positive cubes.
\end{abstract}
\maketitle
\section{Introduction}
It is often the case in additive number theory that there might be problems involving the representation of integers that remain open, yet it can be shown that almost all integers have a representation. Lagrange's celebrated theorem, proven in 1770, states that every positive integer $n$ can be written as
\begin{equation}\label{sqa}n=x_{1}^{2}+x_{2}^{2}+x_{3}^{2}+x_{4}^{2},\end{equation} where $x_{i}\in\mathbb{N}\cup\{0\}.$ Let $\mathscr{C}$ denote the set of integers represented as sums of three positive cubes. In this memoir we will focus our attention on the problem of solving equation (\ref{sqa}) where the set of variables lies on the set $\mathscr{C}$. 

Not very much is known about $\mathscr{C}$. In fact, it isn't even known whether it has positive density or not, the best current lower bound on the cardinality of the set being $$\mathcal{N}(X)=\lvert\mathscr{C}\cap [1,X]\rvert\gg X^{\beta-\varepsilon},$$ where  
$\beta=0.91709477,$ due to Wooley \cite{Woo3}. Under some unproved assumptions on the zeros of some Hasse-Weil $L$-functions, Hooley (\cite{Hol1}, \cite{Hol2}) and Heath-Brown \cite{Hea} showed using different procedures that
$$\sum_{n\leq X}r_{3}(n)^{2}\ll X^{1+\varepsilon},$$ where $r_{3}(n)$ is the number of representations of $n$ as a sum of three positive integral cubes, which implies by applying a standard Cauchy-Schwarz argument that $\mathcal{N}(X)\gg X^{1-\varepsilon}.$
This lack of understanding of the cardinality of the set also prevents us from understanding its distribution over arithmetic progressions, which often comes into play on the major arc analysis. In this memoir, though, we use the classical approach for dealing with exceptional sets involving Bessel's inequality and we make use of an estimate for the minor arcs obtained in forthcoming work of the author \cite{Pli1} to prove that for almost every positive integer $n$ the equation (\ref{sqa}) has a solution with $x_{i}\in\mathscr{C}.$ More precisely, let $E(N)$ be the number of positive integers $n\leq N$ for which (\ref{sqa}) fails to possess a solution with $x_{i}\in\mathscr{C}$. 
\begin{thm}\label{thm01}
For each $\varepsilon>0$ one has \begin{equation}\label{EN}E(N)\ll N(\log N)^{-4/31+\varepsilon}.\end{equation}
\end{thm}
The reader might want to observe that when $n=2^{6+12j}$ for $j\geq 0$ then by taking the equation modulo powers of $2$ one finds that the only solution to (\ref{sqa}) is $x_{i}=2^{2+6j}$. Therefore, $x_{i}\equiv 4\mmod{9}$, and hence $x_{i}$ cannot lie on $\mathscr{C}.$ In the above theorem we are far from obtaining an upper bound of the expected size, but as shown on the preceeding discussion, the exceptional set of integers not represented as in (\ref{sqa}) has infinite cardinality, and in fact
\begin{equation*}E(N)\gg \log N.\end{equation*}

By using the natural polynomial structure given by $\mathscr{C}$ and an estimate for a mean value of some weighted exponential sums, the author proves in \cite{Pli1} via an application of the Hardy-Littlewood method that every sufficiently large integer $n$ can be represented as
$$n=\sum_{i=1}^{8}x_{i}^{2}$$ with $x_{i}\in\mathscr{C}.$ 
In the setting of this paper, the constraint that prevents us from taking fewer variables is the analysis of the minor arcs. Such analysis is based on the use of non-optimal estimates of sums of the shape \begin{equation*}\sum_{m\leq X}a_{m}^{2},\ \ \ \ \ \ \ \ \text{where}\ \ a_{m}=\Big\{\mathbf{x}\in\mathbb{N}^{3}:\ m=x_{1}^{3}+x_{2}^{3}+x_{3}^{3},\ \ x_{2},x_{3}\in \mathcal{A}(P,P^{\eta})\Big\}\end{equation*} with $\eta>0$ a small enough parameter and
$$\mathcal{A}(Y,R)=\{n\in [1,Y]\cap \mathbb{N}: p\mid n\text{ and $p$ prime}\Rightarrow p\leq R\}.$$ Here, the reader may find it useful to observe that it is a consequence of Vaughan and Montgomery \cite[Theorem 7.2]{Mon} that \begin{equation}\label{smo}\text{card}\big(\mathcal{A}(P,P^{\eta})\big)=c_{\eta}P+O\big(P/\log P\big)\end{equation} for some constant $c_{\eta}>0$ that only depends on $\eta$. 

In order to prove Theorem \ref{thm01} we show that for almost every integer the minor arc contribution is of smaller size than the expected main term. We also approximate the generating exponential sums of the problem by an auxiliary function over a set of narrower major arcs. Finally, we show via Bessel's inequality that for almost all integers the $n$-th Fourier coefficient of such exponential sums can also be approximated on the wider major arcs. As experts will realise, the power of $\log N$ saved in (\ref{EN}) comes from the choice of the narrower major arcs and the fact that the error term in (\ref{smo}) only saves a factor of $\log N$. Without severely complicating the argument, this choice seems inevitable for exploiting the information given by the variables $x_{2}$ and $x_{3}$ lying on $\mathcal{A}(P,P^{\eta})$ to ensure the convergence of the singular series and to obtain suitable properties for it. Therefore, the power saving for the bound of the cardinality of the exceptional set seems out of reach with these methods.

The application of Bessel's inequality for bounding exceptional sets has already been used by some authors before (see for instance Montgomery and Vaughan \cite{MonVa}). There is another approach by Wooley which instead uses an exponential sum over the exceptional set that often gives stronger upper bounds for the cardinality of those sets (see Wooley \cite{Woo7}, \cite{Woo8}). However, in order to be able to use the latter method, one would need stronger  minor arc bounds for auxiliary $8$-th moments together with near optimal bounds for $a_{m}$, which are not available in the literature so far.

We devote the rest of the discussion to introduce a harder version of the problem studied here. It is well-known that the numbers that cannot be written as sums of three squares are the ones of the shape $4^{\nu}\cdot m$ for $m\equiv 7\mmod{8}$. Let $$\mathcal{N}=\big\{n\in\mathbb{N}:\ n\not \equiv 7\mmod{9},\ \ \ \ n\neq 4^{\nu}\cdot m\ \text{for some}\ m\equiv 7\mmod{8},\ \ \nu\geq 0\big\}.$$ Then one would hope to have for almost all integers $n\in\mathcal{N}$ a representation
$$n=x_{1}^{2}+x_{2}^{2}+x_{3}^{2}$$ for some $x_{i}\in\mathscr{C}.$ 
If we seek to prove this statement using the circle method approach then one should be able to obtain good enough minor arc bounds of moments of exponential sums involving six variables. We remind the reader though that we are just able to deal with minor arc bounds when we have eight or more variables, and it seems out of reach to lower that number down to $6$. Likewise, the analysis of the singular series with just three variables looks very challenging.

\textbf{Acknowledgements}: The author's work was supported in part by a European Research Council Advanced
Grant under the European Union’s Horizon 2020 research and innovation programme via grant agreement No. 695223 during his studies at the University of Bristol. It was completed while the author was visiting Purdue University under Trevor Wooley's supervision. The author would like to thank him for his guidance and helpful comments, the referee for useful remarks and both the University of Bristol and Purdue University for their support and hospitality.

\section{Notation and preliminary definitions} Unless specified, any lower case letter $\mathbf{x}$ written in bold will denote a triple of integers $(x_{1},x_{2},x_{3})$. We will write $a\leq \mathbf{V}\leq b$ when $a\leq v_{i}\leq b$ for $1\leq i\leq n$. As usual in analytic number theory, for each $x\in\mathbb{R}$ we denote $\exp(2\pi i x)$ by $e(x)$, and for each natural number $q$ then $e(x/q)$ will be written as $e_{q}(x).$ For any scalar $\lambda$ and any vector $\mathbf{x}$ we write $\lambda \mathbf{x}$ for the vector $(\lambda x_{1},\lambda x_{2}, \lambda x_{3})$. Let $N$ be a natural number and consider the parameters $$P=\lfloor N^{1/6}\rfloor,\ \ \ \ M=P^{2/5},\ \ \ \ H=P^{9/5}.$$Observe that then one has $M^{3}H=P^{3}.$ Consider as well$$H_{1}=\Big(\frac{1}{2}\Big)^{1/3}H^{1/3},\ \ \ \ H_{2}=\Big(\frac{2}{3}\Big)^{1/3}H^{1/3},\ \ \ \ H_{3}=\Big(\frac{1}{6}\Big)^{1/3}H^{1/3}.$$ For any vector $\mathbf{x}\in \mathbb{R}^{3}$ set the function $T(\mathbf{x})=x_{1}^{3}+x_{2}^{3}+x_{3}^{3},$ which will be used throughout the paper.
Take the sets of triples
$$\mathcal{H}=\Big\{\mathbf{y}\in\mathbb{N}^{3}:\ \ P/2< y_{1}\leq P,\ \ \ (y_{2},y_{3})\in\mathcal{A}(P,P^{\eta})^{2}\Big\},$$ $$\mathcal{W}=\Big\{\mathbf{y}\in\mathbb{N}^{3}:\ \ H_{1}< y_{1}\leq H_{2},\ \ (y_{2},y_{3})\in\mathcal{A}(H_{3},P^{\eta})^{2}\Big\},$$ where $\eta$ is a sufficiently small but positive parameter. Let $n\in\mathbb{N}$ such that $N/2\leq n\leq N$. We define $R(n)$ as the number of solutions of the equation
$$n=T(p_{1}\mathbf{x}_{1})^{2}+T(p_{2}\mathbf{x}_{2})^{2}+T(\mathbf{x}_{3})^{2}+T(\mathbf{x}_{4})^{2},$$ where 
$\mathbf{x}_{1},\mathbf{x}_{2}\in \mathcal{W},$ $\mathbf{x}_{3},\mathbf{x}_{4}\in \mathcal{H}$ and $M/2\leq p_{1},p_{2}\leq M.$ Our goal in the next sections will be to obtain a lower bound for $R(n)$ for almost all natural numbers. For such purpose, it is convenient to define the weights $$a_{x}=\Big\lvert\big\{\mathbf{y}\in\mathcal{H}:\ \ x=T(\mathbf{y})\big\}\Big\rvert\ \ \ \ \text{and}\ \ \ \ b_{h}=\Big\lvert\big\{\mathbf{y}\in\mathcal{W}:\ \ h=T(\mathbf{y})\big\}\Big\rvert,$$
and consider the exponential sums
$$h(\alpha)=\sum_{x\leq 3P^{3}}a_{x}e(\alpha x^{2})\ \ \ \ \ \ \ \  \text{and} \ \ \ \ \ \ \ \ \ W(\alpha)=\sum_{M/2\leq p\leq M}\sum_{\frac{H}{2}\leq h\leq H}b_{h}e(\alpha p^{6}h^{2}).$$ Observe that by orthogonality it follows that
$$R(n)=\int_{0}^{1}h(\alpha)^{2}W(\alpha)^{2}e(-\alpha n)d\alpha.$$

We will make use of two Hardy-Littlewood dissections in our analysis, and these we now describe. Let $1\leq X\leq P^{4/5}$. When $a\in\mathbb{Z}$ and $q\in\mathbb{N}$ satisfy $0\leq a\leq q\leq X$ and $(a,q)=1$, consider 
\begin{equation*}\grM(a,q)=\Big\{ \alpha\in [0,1): \Big\lvert \alpha-a/q\Big\rvert \leq \frac{X}{qn}\Big\}.\end{equation*} We take the major arcs $\grM(X)$ to be the union of the arcs $\grM(a,q)$. For the sake of simplicity we write
$$\grM=\grM(P^{4/5}),\ \ \ \ \ \ \ \ \grN=\grM\big((\log P)^{\tau}\big),$$ where $\tau=18/31.$ We also define the minor arcs as $\grm=[0,1)\setminus \grM$.

Next we introduce the auxiliary functions that play a leading role in the discussion of Sections \ref{sec3} to \ref{sec6}. For $a\in\mathbb{Z}$ and $q\in\mathbb{N}$ with $(a,q)=1$, let $S(q,a)$ denote the complete exponential sum associated to the problem, which we define by
$$S(q,a)=\sum_{\mathbf{r}\leq q}e_{q}\big(aT(\mathbf{r})^{2}\big).$$ Consider the functions
\begin{equation*}v(\beta)=\int_{\mathbf{x}\in\mathcal{S}}e\big(F(\mathbf{x})\big)d\mathbf{x}\ \ \ \ \ \ \text{and}\ \ \ \ \ \ v_{p}(\beta)=\int_{\mathbf{x}\in \mathcal{S}_{W}}e\big(F_{p}(\mathbf{x})\big)d\mathbf{x},\end{equation*}where $F(\mathbf{x})=\beta T(\mathbf{x})^{2}$ and $F_{p}(\mathbf{x})=\beta T(p\mathbf{x})^{2},$ and the sets of integration taken are $$\mathcal{S}=\Big\{\mathbf{x}\in [0,P]^{3}:\ \ P/2\leq x_{1}\leq P\Big\},\ \ \ \ \mathcal{S}_{W}=\Big\{\mathbf{x}\in\mathbb{R}^{3}:\ \ H_{1}\leq x_{1}\leq H_{2},\ \ 0\leq x_{2}, x_{3}\leq H_{3}\Big\}.$$ Let $\alpha\in [0,1)$ and choose $\beta=\alpha-a/q$. Recalling the constant $c_{\eta}$ mentioned in (\ref{smo}), define  \begin{equation}\label{Vup}V(\alpha,q,a)=q^{-3}S(q,a)c_{\eta}^{2}v(\beta)\ \ \ \ \ \ \text{and}\ \ \ \ \ \ W(\alpha,q,a)=\sum_{M/2\leq p\leq M}V_{p}(\alpha,q,a),\end{equation}where $V_{p}(\alpha,q,a)=q^{-3}S(q,a)c_{\eta}^{2}v_{p}(\beta).$ For the sake of brevity, we define the auxiliary functions $h^{*}(\alpha)$ and $W^{*}(\alpha)$ by setting
\begin{equation*}h^{*}(\alpha)=V(\alpha,q,a)\ \ \ \ \ \text{and}\ \ \ \ \ W^{*}(\alpha)=W(\alpha,q,a) \end{equation*} when $\alpha\in\grM(a,q)\subset\grM$ and $h^{*}(\alpha)=W^{*}(\alpha)=0$ for $\alpha\in\grm.$
Before describing the outline of the memoir, it is convenient to introduce $$\mathcal{F}(\alpha)=h(\alpha)^{2}W(\alpha)^{2}-h^{*}(\alpha)^{2}W^{*}(\alpha)^{2}.$$

In Section \ref{sec3} we approximate $h(\alpha)$ and $W(\alpha)$ by  the functions $h^{*}(\alpha)$ and $W^{*}(\alpha)$ respectively when $\alpha\in\grN$. Making use of these approximations we bound the integral of $\mathcal{F}(\alpha)e(-\alpha n)$ over $\grN$ in Section \ref{sec5}. In the second part of that section and Section \ref{sec6} we show that the upper bound for the integral of the same function over $\grM\setminus\grN$ still holds for almost all integers. We obtain such result by estimating the integral of $\lvert\mathcal{F}(\alpha)\rvert^{2}$ and applying Bessel's inequality. Section \ref{sec4} is devoted to the study of the singular series. In such analysis we give a lower bound of the singular series for almost all integers, which combined with the lower bound for the singular integral computed in Section \ref{sec6} provides a lower bound for the major arc contribution. We also combine the major arc estimates obtained throughout the memoir with a result in forthcoming work of the author \cite{Pli1} to show in Section \ref{sec6} that the minor arc contribution is smaller than the major arc one for almost all integers.

Whenever $\varepsilon$ appears in any bound, it will mean that the bound holds for every $\varepsilon>0$, though the implicit constant then may depend on $\varepsilon$. We adopt the convention that when we write $\delta$ in the computations we mean that there exists a positive constant such that the bound holds. We use $\ll$ and $\gg$ to denote Vinogradov's notation, and write $A\asymp B$ whenever $A\ll B\ll A$. We write $p^{r}|| n$ to denote that $p^{r}| n$ but $p^{r+1}\nmid n.$
\section{Approximation of exponential sums over the major arcs.}\label{sec3}
Based on previous work by the author we briefly provide some technical lemmas to approximate the exponential sums over the set of narrower major arcs.

\begin{lem}\label{lem1}Let $\alpha\in \grN(a,q)$ with $a\in\mathbb{Z}$, $q\in\mathbb{N}$ and $(a,q)=1$. Then one obtains the formula
$$h(\alpha)=V(\alpha,q,a)+O\big(P^{3}(\log P)^{\tau-1+\varepsilon}\big).$$
\end{lem}
\begin{proof}
This is a consequence of Lemma 7.3 of \cite{Pli}. The reader may check that the exponential sum $h(\alpha)$ here corresponds to $g_{Q,m}(\alpha)$ with the choices $Q=P,$ $m=1$ and constants $C_{1}=1/2$, $C_{2}=1$ and $C_{3}=1.$
\end{proof}

\begin{lem}\label{lem3}Let $\alpha\in \grN(a,q)$ with $a\in\mathbb{Z}$, $q\in\mathbb{N}$ and $(a,q)=1$. Then,
$$W(\alpha)=W(\alpha,q,a)+O\big(HM(\log P)^{\tau-2+\varepsilon}\big).$$
\end{lem}
\begin{proof}
Observe that $$W(\alpha)=\sum_{M/2\leq p\leq M}W_{p}(\alpha),\ \ \ \ \ \ \text{where}\ \ W_{p}(\alpha)=\sum_{\mathbf{x}\in\mathcal{W}}e\big(\alpha T(p\mathbf{x})^{2}\big).$$
We also apply Lemma 7.3 of \cite{Pli} to $W_{p}(\alpha)$. The reader may check that $W_{p}(\alpha)$ here corresponds to $g_{Q,m}(\alpha)$ with the choices $Q=H^{1/3}$ and $m=p$ and the constants $C_{1}=\big(1/2\big)^{1/3}$, $C_{2}=\big(2/3\big)^{1/3}$ and $C_{3}=\big(1/6\big)^{1/3}.$ Consequently, it transpires that $$W_{p}(\alpha)=V_{p}(\alpha,q,a)+O\big(H\big(\log P)^{\tau-1+\varepsilon}\big),$$
which delivers the result.
\end{proof}
The following series of lemmas make use of the work done in forthcoming publications of the author (\cite {Pli} and \cite{Pli1}) to give upper bounds for the auxiliary functions that play a role in the main term of the contribution of the major arcs.
\begin{lem}\label{lema4}
Let $ \beta \in\mathbb{R}.$ Then one has that
$$v(\beta)\ll \frac{P^{3}}{1+n\lvert \beta\rvert}\ \ \ \ \ \text{and}\ \ \ \ \ v_{p}(\beta)\ll \frac{H}{1+n\lvert \beta\rvert}.$$ 
\end{lem}
\begin{proof}
For each $\mathbf{y}\in\mathbb{R}^{2},$ consider $C_{\mathbf{y}}=y_{1}^{3}+y_{2}^{3}$. Define the auxiliary functions $$v_{\mathbf{y}}(\beta)=\int_{M_{\mathbf{y}}}^{N_{\mathbf{y}}}B_{\mathbf{y}}(\gamma)e(\beta \gamma)d\gamma\ \ \ \ \ \text{and}\ \ \ \ \ v_{\mathbf{y},p}(\beta)=\int_{M_{\mathbf{y},p}}^{N_{\mathbf{y},p}}B_{\mathbf{y},p}(\gamma)e(\beta \gamma)d\gamma$$ where $B_{\mathbf{y}}(\gamma)$, $B_{\mathbf{y},p}(\gamma)$ and the limits of integration taken are$$B_{\mathbf{y}}(\gamma)=\frac{1}{6}\gamma^{-1/2}\big(\gamma^{1/2}-C_{\mathbf{y}}\big)^{-2/3},\ \ \  M_{\mathbf{y}}=\Big(\frac{P^{3}}{8}+C_{\mathbf{y}}\Big)^{2},\ \ \ N_{\mathbf{y}}=\big(P^{3}+C_{\mathbf{y}}\big)^{2}$$and
$$B_{\mathbf{y},p}(\gamma)=\frac{1}{6p}\gamma^{-1/2}\big(\gamma^{1/2}-C_{p\mathbf{y}}\big)^{-2/3},\ \ M_{\mathbf{y},p}=\Big(\frac{p^{3}H}{2}+C_{p\mathbf{y}}\Big)^{2} \ \text{and}\ \ N_{\mathbf{y},p}=\Big(\frac{2p^{3}H}{3}+C_{p\mathbf{y}}\Big)^{2}.$$
By a change of variables we find that
\begin{equation}\label{vbeta}v(\beta)=\int_{\mathbf{y}\in [0,P]^{2}}v_{\mathbf{y}}(\beta)d\mathbf{y}\ \ \ \ \ \ \text{and}\ \ \ \ \ \ \ v_{p}(\beta)=\int_{\mathbf{y}\in [0,H_{3}]^{2}}v_{\mathbf{y},p}(\beta)d\mathbf{y}.\end{equation}
Observe that Lemma 5.2 of \cite{Pli1} yields the pointwise bounds $v_{\mathbf{y}}(\beta)\ll P(1+n\lvert\beta\rvert)^{-1}$ and $v_{\mathbf{y},p}(\beta)\ll H^{1/3}(1+n\lvert\beta\rvert)^{-1}.$ The result then follows applying these estimates trivially to the above integrals. 
\end{proof}
\begin{lem}\label{lem0}
Let $a\in\mathbb{Z}$ and $q\in\mathbb{N}$ with $(a,q)=1.$ Then, one has
\begin{equation}\label{kkk}S(q,a)\ll q^{5/2+\varepsilon}.\end{equation} Moreover, when $p$ is prime and $l\geq 3$ one finds that \begin{equation}\label{kk}S(p^{l},a)\ll lp^{5l/2+\varepsilon}.\end{equation} When $l=2$ then $S(p^{2},a)\ll p^{5}$ and for the case $l=1$ we obtain the refinement
\begin{equation}\label{Sap}S(p,a)=p^{2}S_{2}(p,a)+O(p^{2}), \ \ \ \ \ \ \ \text{where}\  S_{2}(p,a)=\sum_{r=1}^{p}e_{p}(ar^{2}).\end{equation}
\end{lem}
\begin{proof}
Equations (\ref{kk}) and (\ref{Sap}) follow from Lemmata 3.1 and 3.2 of \cite{Pli} respectively. The bound for the case $l=2$ also follows from Lemma 3.1 of \cite{Pli}. We remind the reader of the estimates $$d(q)\ll q^{\varepsilon},\ \ \ \ \ \ \ \ \omega(q)\ll \log q/\log\log q,$$ where the functions $d(q)$ and $\omega(q)$ denote the number of divisors of $q$ and the number of prime divisors of $q$, respectively. Combining such bounds with the multiplicative property of $S(q,a)$ and the estimates for $S(p^{l},a)$ discussed before we obtain (\ref{kkk}).
\end{proof}The next lemma gathers the previous results to provide an upper bound for the auxiliary functions $h^{*}(\alpha)$ and $W^{*}(\alpha)$.
\begin{lem}\label{lem06}
Let $a\in\mathbb{Z}$ and $q\in\mathbb{N}$ with $(a,q)=1$. Take $\alpha\in \grM(a,q)$. Then,
$$h^{*}(\alpha)\ll \frac{q^{-1/2+\varepsilon}P^{3}}{1+n\lvert\beta\rvert}\ \ \ \ \ \ \text{and}\ \ \ \ \ \ W^{*}(\alpha)\ll \frac{q^{-1/2+\varepsilon}HM}{(\log P)(1+n\lvert\beta\rvert)}.$$
\end{lem}
\begin{proof}
This follows from Lemmata \ref{lema4} and \ref{lem0} via equation (\ref{Vup}).
\end{proof}

\section{Treatment of the singular series}\label{sec4}
In this section we discuss some convergence properties of the singular series and we analyse the local solubility of the problem. As a consequence, we derive a lower bound for the singular series for almost all integers. For such purposes, it is convenient to define, for $q\in\mathbb{N}$, the sums
$$S_{n}(q)=\sum_{\substack{a=1\\ (a,q)=1}}^{q}\big(q^{-3}S(q,a)\big)^{4}e_{q}(-na),\ \ \ \ \ \ \ \ \mathfrak{S}(n)=\sum_{q=1}^{\infty}S_{n}(q)$$and for each prime $p$, the infinite series
$$\sigma(p)=\sum_{l=0}^{\infty}S_{n}(p^{l}).$$
\begin{lem}\label{lema3}
One has that
\begin{equation*}\mathfrak{S}(n)=\prod_{p}\sigma(p),\end{equation*}
the singular series $\mathfrak{S}(n)$ converges absolutely and $\mathfrak{S}(n)\ll n^{\varepsilon}$. Also, when $Q> 0$ one gets
\begin{equation}\label{tup}\sum_{q\leq Q}q^{1/2}\lvert S_{n}(q)\rvert\ll (nQ)^{\varepsilon}\ \ \ \ \ \ \ \text{and}\ \ \ \ \ \ \ \ \sum_{q\geq Q}\lvert S_{n}(q)\rvert\ll n^{\varepsilon}Q^{\varepsilon-1/2}.\end{equation}Moreover, for any constant $\upsilon>0$ there exists a set $A_{\upsilon}\subset [1,N]$ with $\lvert A_{\upsilon}\rvert\ll N(\log N)^{-\upsilon}$ and such that for every $n\in [1,N]\setminus A_{\upsilon}$ one obtains
$$\mathfrak{S}(n)\gg (\log N)^{-\upsilon}.$$ 
\end{lem}
\begin{proof}
Recalling (\ref{Sap}), consider the exponential sum
$$S_{n,2}(q)=\sum_{\substack{a=1\\ (a,q)=1}}^{q}\big(q^{-1}S_{2}(q,a)\big)^{4}e_{q}(-na).$$ By Lemma \ref{lem0} one has that \begin{equation}\label{Sp}S_{n}(p)=S_{n,2}(p)+O(p^{-3/2}).\end{equation}Observe that an application of the same lemma yields $S_{n}(p^{l})\ll l^{4}p^{-l+\varepsilon}$ when $l\geq 3$ and $S_{n}(p^{2})\ll p^{-2}$. One can also deduce from equation (4.27) of Vaughan \cite[Theorem 4.3]{Vau} that whenever $p\nmid n$ then $S_{n,2}(p)\ll p^{-3/2}.$ Therefore, the combination of the previous estimates gives
 $$\displaystyle\sum_{l=1}^{\infty}\lvert S_{n}(p^{l})\rvert\ll p^{-3/2}.$$ Likewise, when $p\mid n,$ then an application of the aforementioned bounds for $S_{n}(p^{l})$ and  the estimate $S_{n,2}(p)\ll p^{-1}$, which is an immediate consequence of Vaughan \cite[Lemma 4.3]{Vau}, lead to $$\displaystyle\sum_{l=1}^{\infty}\lvert S_{n}(p^{l})\rvert\ll p^{-1}.$$ Therefore, by the preceeding discussion and the multiplicative property of $S_{n}(q)$, one gets the convergence for $\frak{S}(n)$ and the upper bound
$$\frak{S}(n)\ll \prod_{p\nmid n}\big(1+C_{1}p^{-3/2}\big)\prod_{p\mid n}\big(1+C_{2}p^{-1}\big)\ll n^{\varepsilon}$$ for some constants $C_{1},C_{2}>0$. Similarly, one finds that
$$\sum_{l=1}^{\infty}p^{l/2}\lvert S_{n}(p^{l})\rvert\ll p^{-\xi},$$ where $\xi=1$ if $p\nmid n$ and $\xi=1/2$ if $p\mid n$. Consequently, the combination of the above estimates yields the bound
$$\sum_{q\leq Q}q^{1/2}\lvert S_{n}(q)\rvert\ll \prod_{\substack{p\leq Q\\ p\nmid n}}(1+C_{1}p^{-1})\prod_{\substack{p\leq Q\\ p|n}}(1+C_{2}p^{-1/2})\ll (nQ)^{\varepsilon}.$$The second estimate in (\ref{tup}) follows observing that as a consequence of the above equation then
$$\sum_{Q\leq q\leq 2Q}\lvert S_{n}(q)\rvert\ll n^{\varepsilon}Q^{\varepsilon-1/2},$$ whence summing over dyadic intervals we obtain the desired result.

We will devote the rest of the section to prove the lower bound for the singular series. By equation (4.27) of Vaughan \cite[Theorem 4.3]{Vau} and (\ref{Sp}) then whenever $p\nmid n$ one has $S_{n}(p)\ll p^{-3/2}$. We can also deduce from the proof\footnote{See the argument just before \cite[Theorem 4.6]{Vau}} of \cite[Theorem 4.5]{Vau} that $S_{n,2}(p)\geq 0$ for $p\mid n$, which, combined with (\ref{Sp}), yields $S_{n}(p)\geq -C_{3}p^{-3/2}$ for some $C_{3}>0$. Consequently, using the bound $S_{n}(p^{l})\ll l^{4}p^{-l+\varepsilon}$ for $l\geq 2$ mentioned after (\ref{Sp}) one gets that in both cases then $\sigma(p)\geq 1-C_{4}p^{-3/2}$ for some $C_{4}>0,$ and hence there exists a constant $C>0$ for which
\begin{equation}\label{prod}\mathfrak{S}(n)\gg \prod_{p\leq C}\sigma(p).\end{equation}In order to give a more arithmetic description of $\sigma(p)$ we define for each $h\in\mathbb{N}$ the set $$\mathcal{M}_{n}(p^{h})=\Big\{\mathbf{Y}\in [1,p^{h}]^{12}:\ \sum_{i=1}^{4}T(\bfy_{i})^{2}\equiv n\mmod{p^{h}}\Big\},$$ and ${M}_{n}(p^{h})=\lvert\mathcal{M}_{n}(p^{h})\rvert$. Observe that orthogonality yields the identity
$$\sum_{l=0}^{h}S_{n}(p^{l})=p^{-11h}M_{n}(p^{h}),$$ and hence $\sigma(p)=\lim_{h\to \infty}p^{-11h}M_{n}(p^{h}).$ For further discussion, it is relevant to introduce the set $$\mathcal{M}_{n}^{*}(p^{h})=\Big\{\mathbf{Y}\in \mathcal{M}_{n}(p^{h}):\ \  p\nmid y_{1,1},\ p\nmid T(\mathbf{y}_{1})\Big\},$$ where $\mathbf{y}_{1}=(y_{1,1},y_{1,2},y_{1,3}),$ and ${M}_{n}^{*}(p^{h})=\lvert\mathcal{M}_{n}^{*}(p^{h})\rvert.$ By Lemma 4.2 of \cite{Pli} we have that whenever $p\neq 2,3$ then $M_{n}^{*}(p)>0$ for all $n$. Consequently, a standard application of Hensel's Lemma leads to $M_{n}(p^{h})\geq p^{11(h-1)}$, which yields $\sigma(p)\geq p^{-11}.$ For the cases $p=2,3$, it is convenient to define the set
\begin{equation*}\label{ec11.0}\mathcal{M}_{3,3}(p^{h})=\Big\{T(\bfx):\ \bfx\in\big(\mathbb{Z}/p^{h}\mathbb{Z}\big)^{3},\ (x_{1},p)=1\Big\}.\end{equation*} A slightly tedious computation reveals that $\mathcal{M}_{3,3}(27)$ is the set of residues not congruent to $4$ or $5$ modulo $9$. Therefore, one has that $$A=\Big\{x^{2}\mmod{27},\ \ \ \ x\in \mathcal{M}_{3,3}(27)\Big\}=\Big\{0,1,4,9,10,13,19,22\Big\}.$$ Consider the set $B=\{y\in A:\ (y,3)=1\big\}.$ Observe that then $$A+B=\Big\{1,2,4,5,8,10,11,13,14,17,19,20,22,23,26\Big\}.$$ Consequently, by Cauchy-Davenport (see \cite[Lemma 2.14]{Vau}) we find that $M_{n}^{*}(27)>0$, whence another application of Hensel's Lemma gives $M_{n}(3^{h})\geq 3^{11(h-3)}$, and hence $\sigma(3)\geq 3^{-33}.$

The rest of the discussion will be devoted to the analysis for the prime $p=2$. Take $\gamma\geq 0$ to be the exponent for which $2^{\gamma}|| n$ and let $\theta=\lfloor (\gamma-1)/2\rfloor$. A routine application of Hensel's Lemma reveals that $\mathcal{M}_{3,3}(2^{h})$ consists of all the residue classes modulo $2^{h}$. Therefore, when $\gamma\leq 2$, one has that $M_{n}^{*}(8)>0$, whence another application of Hensel's Lemma would yield $\sigma(2)\geq 2^{-33}.$ For the case $\gamma\geq 3$ then whenever $h\geq \gamma+2$ one can check that the congruence 
\begin{equation}\label{squares}x_{1}^{2}+x_{2}^{2}+x_{3}^{2}+x_{4}^{2}\equiv n\mmod{2^{h}}\end{equation} is soluble with solutions $x_{i}=2^{\theta}y_{i},$ where $y_{i}$ is defined modulo $2^{h-\theta}$ and 
\begin{equation}\label{mod8}y_{1}^{2}+y_{2}^{2}+y_{3}^{2}+y_{4}^{2}\equiv 2^{-2\theta}n \mmod{2^{h-2\theta}}\end{equation}
with $2\nmid y_{1}$. Note that (\ref{mod8}) has a solution modulo $8$ with $2\nmid y_{1}$, and hence by Lemma 2.13 of \cite{Vau} there are at least  $2^{3(h-2\theta-3)}\times 2^{4\theta}$ solutions to (\ref{squares}). By the same lemma, one has that the number of solutions to $$z_{1}^{3}+z_{2}^{3}+z_{3}^{3}\equiv x_{i}\mmod{2^{h}}$$ is bounded below by $2^{2(h-1)}$. Consequently, we obtain $M_{n}(2^{h})\geq 2^{11h-\gamma-16}$, which delivers $\sigma (2)\gg 2^{-\gamma}.$

To finish the proof we take $A_{\upsilon}\subset [1,N]$ to be the set of numbers with $2^{\gamma}\geq (\log N)^{\upsilon}$. Observe that $\lvert A_{\upsilon}\rvert \leq N(\log N)^{-\upsilon}$. Then, by the preceeding discussion and (\ref{prod}) it follows that whenever $n\notin A_{\upsilon}$ one has
$$\mathfrak{S}(n)\gg (\log N)^{-\upsilon}.$$

\end{proof}

\section{Mean values of the error term over the major arcs}\label{sec5}
Before proving an estimate for the first and second moment of $\mathcal{F}(\alpha)$ over $\grN$ and $\grM\setminus \grN$ respectively we will present some major arc type bounds that will be used later on the proof. For such matters, it is convenient to introduce the auxiliary multiplicative function $w_{2}(q)$, defined for prime powers by taking
\begin{equation}\label{wuok}w_{2}(p^{6u+v})=\left\{
	       \begin{array}{ll}
         p^{-u-v/6}\ \ \ \ \ \ \ \text{when $u\geq 1$ and $1\leq v\leq 6$}   \\
         p^{-1}\ \ \ \ \ \ \ \  \ \ \ \ \ \text{when $u=0$ and $2\leq v\leq 6$}    \\
         p^{-1/2}\ \ \ \ \ \ \  \ \ \ \ \text{when $u=0$ and $v=1$.}    \\
    \end{array}  \right. \end{equation}
\begin{lem}\label{prop4}
Let $a\in\mathbb{Z}$ and $q\in\mathbb{N}$ with $(a,q)=1$ and take $\alpha\in\grM(a,q)$. Denote $\beta=\alpha-a/q.$ Then, 
$$h(\alpha)\ll \frac{q^{\varepsilon}w_{2}(q)P^{3}}{1+n\lvert \beta\rvert},$$ and for $\lvert\beta\rvert\leq (12q)^{-1}H^{1/3}n^{-1}$ and $q\leq M/4$ we have that
\begin{equation}\label{lam}W(\alpha)\ll \frac{q^{\varepsilon}w_{2}(q)HM}{(\log P)(1+n\lvert \beta\rvert)}.\end{equation}
\end{lem}
\begin{proof}
Lemmata 3.1 and 5.2 of \cite{Pli1} yield the bounds
$$h(\alpha)\ll q^{\varepsilon}w_{2}(q)P^{3}(1+n\lvert\beta\rvert)^{-1}+P^{2}q^{1+\varepsilon}w_{2}(q)$$whenever $\alpha\in\grM(a,q).$ Observe that $(1+n\lvert\beta\rvert)^{-1}\geq qP^{-1}$ when $\alpha\in\grM(a,q)$, and so the first term on the right side of the above equation dominates over the second one. Likewise, Lemmata 3.2 and 5.2 of \cite{Pli1} deliver $$W(\alpha)\ll q^{\varepsilon}w_{2}(q)MH(\log P)^{-1}(1+n\lvert \beta\rvert)^{-1}+MH^{2/3}q^{1+\varepsilon}w_{2}(q)(\log P)^{-1}$$ for the range described just before (\ref{lam}). Noting that $(1+n\lvert\beta\rvert)^{-1}\geq qH^{-1/3}$ we find that the first term also dominates over the second one in the above equation. The preceeding discussion then provides the lemma.
\end{proof}

\begin{lem}\label{flo}
For $q\in\mathbb{N}$ and every $Q>0$ one finds that $w_{2}(q)\leq q^{-1/6}$. Moreover, one has \begin{equation*}\sum_{q\leq Q}w_{2}(q)^{2}\ll Q^{\varepsilon}\ \ \ \ \ \ \ \ \ \ \  \text{and}\ \ \ \ \ \ \ \ \ \ \  \sum_{q\leq Q}w_{2}(q)^{2+\delta}\ll 1\end{equation*} for any $\delta>0$.
\end{lem}
\begin{proof}
Both estimates follow from the definition (\ref{wuok}) and the fact that $w_{2}(q)$ is multiplicative. 
\end{proof}
Combining the previous technical lemmas, we provide bounds for the $L^{1}$-norm of $\mathcal{F}(\alpha)$ over the set of arcs $\grN$ and the $L^{2}$-norm of the same function over $\grM\setminus\grN$ which are good enough for our purposes.
\begin{prop}\label{prop11}
One has that
\begin{equation}\label{tus}\int_{\grN}\big\lvert \mathcal{F}(\alpha)\big\rvert d\alpha\ll (HM)^{2}(\log P)^{3\tau/2-3+\varepsilon}.
\end{equation}
\end{prop}
\begin{proof}
Recalling Lemmata \ref{lem1} and \ref{lem06} it follows that whenever $\alpha\in\grN(a,q)\subset \grN$ then
\begin{equation}\label{hal}h(\alpha)^{2}-h^{*}(\alpha)^{2}\ll P^{6}(\log P)^{\tau-1+\varepsilon}\Big((\log P)^{\tau-1}+q^{-1/2}(1+n\lvert\beta\rvert)^{-1}\Big).\end{equation} Likewise, by Lemmata \ref{lem3} and \ref{lem06} we have that for $\alpha\in\grN(a,q)\subset\grN$ then
\begin{equation}\label{Wal}W(\alpha)^{2}-W^{*}(\alpha)^{2}\ll (HM)^{2}(\log P)^{\tau-3+\varepsilon}\Big((\log P)^{\tau-1}+q^{-1/2}(1+n\lvert\beta\rvert)^{-1}\Big).\end{equation}
Denote by $N(q)$ the number of solutions of the congruence 
$$T(p_{1}\mathbf{x}_{1})^{2}\equiv T(p_{2}\mathbf{x}_{2})^{2}\mmod {q},$$ where $\mathbf{x}_{i}\in [1,H^{1/3}]^{3}$ and $M/2\leq p_{i}\leq M$. By expressing $q$ as the product of prime powers, using the structure of the ring of integers of those prime powers and noting that the number of primes dividing $q$ is bounded by $q^{\varepsilon}$, we obtain
\begin{equation*}\label{Nq}N(q)\ll q^{\varepsilon-1}(HM)^{2}(\log P)^{-2},\end{equation*}and hence orthogonality delivers \begin{equation}\label{Wint}\sum_{a=1}^{q}\lvert W(\beta+a/q)\rvert^{2}\leq qN(q)\ll q^{\varepsilon}(HM)^{2}(\log P)^{-2}.\end{equation}
Integrating over the major arcs and applying (\ref{hal}) and (\ref{Wint}) one gets
\begin{align*}\label{ugu}&\int_{\grN}\big\lvert h(\alpha)^{2}-h^{*}(\alpha)^{2}\big\rvert \lvert W(\alpha)\rvert^{2}d\alpha\ll (HM)^{2}(\log P)^{3\tau-4+\varepsilon}\sum_{q\leq (\log P)^{\tau}}q^{-1}
\\
&+(HM)^{2}(\log P)^{\tau-3+\varepsilon}\sum_{q\leq (\log P)^{\tau}}q^{-1/2}\ll (HM)^{2}(\log P)^{3\tau/2-3+\varepsilon}.\nonumber
\end{align*}
Similarly, using  Lemma \ref{lem06} and (\ref{Wal}) we obtain
\begin{align*}
\int_{\grN}\lvert h^{*}(\alpha)\rvert ^{2}\big\lvert W(\alpha)^{2}-W^{*}(\alpha)^{2}\big\rvert d\alpha\ll (HM)^{2}(\log P)^{3\tau/2-3+\varepsilon}.
\end{align*}
Equation (\ref{tus}) then holds combining the previous estimates and the triangle inequality.
\end{proof}
Observe that the error term in Lemmata \ref{lem1} and \ref{lem3} when we approximate $h(\alpha)$ and $W(\alpha)$ by $h^{*}(\alpha)$ and $W^{*}(\alpha)$ respectively is non-trivial only for the set of small major arcs $\grN$. For the wider major arcs, we obtain instead an almost all result via Bessel's inequality.
\begin{prop}\label{lem7}
One has that
\begin{equation*}\int_{\grM\setminus \grN}\big\lvert \mathcal{F}(\alpha)\big\rvert^{2}d\alpha\ll P^{6}(HM)^{4}(\log P)^{-4-2\tau/3+\varepsilon}.\end{equation*}
\end{prop}
\begin{proof}
Note that $\lvert \mathcal{F}(\alpha)\rvert^{2}\ll \lvert h(\alpha)\rvert^{4}\lvert W(\alpha)\rvert^{4}+\lvert h^{*}(\alpha)\rvert^{4}\lvert W^{*}(\alpha)\rvert^{4}.$
For bounding the above integral we make use of standard major arc techniques and we exploit the extra number of variables that we get by taking squares. Before going into the proof, it is convenient to define $\Upsilon_{\varepsilon}(\alpha)$ for $\alpha\in [0,1)$ and $\varepsilon>0$ by taking
$$\Upsilon_{\varepsilon}(\alpha)=q^{\varepsilon}w_{2}(q)(1+n\lvert\alpha-a/q\lvert)^{-1}$$ when $\alpha\in\grM(a,q)\subset\grM$ and $\Upsilon_{\varepsilon}(\alpha)=0$ otherwise. Using Lemma \ref{prop4} we find that
$$\int_{\grM\setminus \grN}\lvert h(\alpha)\rvert^{4}\lvert W(\alpha)\rvert^{4}d\alpha\ll P^{12}\int_{\grM\setminus \grN}\lvert W(\alpha)\rvert^{4}\Upsilon_{\varepsilon}(\alpha)^{4}d\alpha.$$ 
Observe that combining Lemma \ref{flo} with equation (\ref{Wint}) we obtain that the contribution of the arcs with $q>M/4$ or $q\leq M/4$ and $\lvert\beta\rvert>  (12q)^{-1}H^{1/3}n^{-1}$ is $O\big((HM)^{4}P^{6-\delta}).$ Let $I'$ be the contribution to $I$ of the arcs with $q\leq M/4$ and $\lvert\beta\rvert\leq (12q)^{-1}H^{1/3}n^{-1}$. Then  Lemma \ref{prop4} yields
\begin{equation*}I'\ll P^{12}(HM)^{2}(\log P)^{-2}\int_{\grM\setminus \grN}\lvert W(\alpha)\rvert^{2}\Upsilon_{\varepsilon}(\alpha)^{6}d\alpha,\end{equation*}
whence using Lemma \ref{flo} and (\ref{Wint}) again we get that $I'=O\big(P^{6}(HM)^{4}(\log P)^{-4-2\tau/3+\varepsilon}\big).$
On the other hand, an application of Lemma \ref{lem06} gives the estimate
\begin{equation*}\int_{\grM\setminus \grN}\lvert h^{*}(\alpha)\rvert^{4}\lvert W^{*}(\alpha)\rvert^{4}\ll P^{6}(HM)^{4}(\log P)^{-4-2\tau+\varepsilon},
\end{equation*}which concludes the proof.
\end{proof}

\section{Singular integral and Proof of Theorem 1.1}\label{sec6}

We briefly introduce the singular integral, give upper and lower bounds for it and discuss the size of the exceptional sets of the integers $n$ with large minor arc contribution and for which the $n$-th Fourier coefficient of $\mathcal{F}(\alpha)$ over $\grM\setminus\grN$ is large as well.
Consider \begin{equation}\label{Jn}J(n)=\sum_{\mathbf{p}}\int_{\mathbf{Y}}J_{\mathbf{Y},\mathbf{p}}(n)d\mathbf{Y},\end{equation} where we define the collection $J_{\mathbf{Y},\mathbf{p}}(n)$ of singular integrals by \begin{equation*}J_{\mathbf{Y},\mathbf{p}}(n)=\int_{-\infty}^{\infty}V_{\mathbf{Y},\mathbf{p}}(\beta)e(-n\beta)d\beta\ \ \ \text{and}\ \ \ \ V_{\mathbf{Y},\mathbf{p}}(\beta)=\prod_{i=1}^{2}v_{\mathbf{y}_{i},p_{i}}(\beta)\prod_{i=3}^{4}v_{\mathbf{y}_{i}}(\beta).\end{equation*}
The range of integration taken above is the set of tuples $\mathbf{Y}=(\mathbf{y}_{1},\ldots,\mathbf{y}_{4})$ with $\mathbf{y}_{1},\mathbf{y}_{2}\in\mathcal{A}(H_{3},P^{\eta})^{2}$ and $\mathbf{y}_{3},\mathbf{y}_{4}\in\mathcal{A}(P,P^{\eta})^{2}$. Likewise, $\mathbf{p}$ runs over pairs of primes $(p_{1},p_{2})$ with $M/2\leq p_{i}\leq M.$  Here the reader might find useful to recall (\ref{vbeta}) and to observe that Lemma \ref{lema4} guarantees the absolute convergence of the above integrals.
\begin{lem}\label{JN}
One has that
$$J(n)\asymp (HM)^{2}(\log N)^{-2}.$$
\end{lem}
\begin{proof}
An inspection of the proof of Lemma 5.1 of \cite{Pli1} reveals that the positivity and the upper bound for $J_{\mathbf{Y},\mathbf{p}}(n)$ deduced there remain valid subject only to the condition $s+t\geq 2$. Consequently, on making the choices $k=s=t=2$ here we obtain $0\leq J_{\mathbf{Y},\mathbf{p}}(n)\ll P^{-4}H^{2/3}$, whence applying this estimate trivially to (\ref{Jn}) gives the required upper bound. Likewise, when $M/2\leq p_{i}\leq 51 M/100$ for $i\leq t$ and $\mathbf{y}_{i}\leq P/2$ for $t+1\leq i\leq s+t$ then the lower bound obtained in that lemma also holds as long as $\big(\big(3/8\big)^{k}s+\big(1/8\big)^{k}t\big)P^{3k}<n$. Therefore, on considering such range here we get that $J_{\mathbf{Y},\mathbf{p}}(n)\gg P^{-4}H^{2/3}$. Observe that the set of tuples on that range has positive density over the set without the restrictions. Consequently, the preceeding remark and the positivity of $J_{\mathbf{Y},\mathbf{p}}(n)$ deliver the lower bound stated at the beginning.
\end{proof}
We remind the reader of the definition (\ref{Vup}). Note that then the combination of Lemma \ref{lema4} and equation (\ref{tup}) with a change of variables yields
\begin{equation}\label{hWW}\int_{\grM}h^{*}(\alpha)^{2}W^{*}(\alpha)^{2}e(-\alpha n)d\alpha=\frak{S}(n)J(n)+O\big((HM)^{2}N^{-\delta}\big).\end{equation}
For the rest of the section we introduce the exceptional sets which arise in both the major and the minor arc analysis and we give bounds for the cardinality of them. Let $\delta>0$ and let $\mathcal{E}_{\delta}(N)$ denote the set of integers $N/2\leq n\leq N$ such that 
$$\int_{\grm} h(\alpha)^{2}W(\alpha)^{2}e(-n\alpha)d\alpha\gg (HM)^{2}P^{-\delta/2}.$$ Likewise, define $\mathcal{E}(N)$ to be the set of integers $N/2\leq n\leq N$ for which
\begin{equation}\label{exc}\int_{\grM\setminus\grN}\mathcal{F}(\alpha)e(-n\alpha)d\alpha\gg (HM)^{2}(\log P)^{-2-2\tau/9}.\end{equation}
\begin{prop}\label{prop1}
With the above notation, one has that $$\lvert \mathcal{E}(N)\rvert\ll N(\log N)^{-2\tau/9+\varepsilon},$$ and there exists some $\delta>0$ for which $\lvert \mathcal{E}_{\delta}(N)\rvert\ll N^{1-\delta}.$
\end{prop}
\begin{proof}
We obtain these two bounds via a routine application of Bessel's inequality. For such matters, observe first that Proposition 1 of \cite{Pli1} gives the estimate $$\int_{\grm}\lvert h(\alpha)W(\alpha)\rvert^{4}d\alpha\ll (HM)^{4}P^{6-2\delta},$$ for some $\delta>0$. Define the Fourier coefficient $c(n)$ of the product of the generating functions on the minor arcs by $$c(n)=\int_{\grm} h(\alpha)^{2}W(\alpha)^{2}e(-n\alpha)d\alpha.$$ Note that Bessel's inequality yields
$$\sum_{N/2\leq n\leq N}\lvert c(n)\rvert^{2}\ll \int_{\grm}\lvert h(\alpha)W(\alpha)\rvert^{4}d\alpha\ll (HM)^{4}P^{6-2\delta},$$whence the bound on $\lvert\mathcal{E}_{\delta}(N)\rvert$ follows from last equation.
Similarly, we introduce the Fourier coefficient $$a(n)=\int_{\grM\setminus\grN}\mathcal{F}(\alpha)e(-n\alpha)d\alpha.$$ Then by Proposition \ref{lem7} and Bessel's inequality we get
$$\sum_{N/2\leq n\leq N}\lvert a(n)\rvert^{2}\ll \int_{\grM\setminus\grN}\lvert \mathcal{F}(\alpha)\rvert^{2}d\alpha\ll P^{6}(HM)^{4}(\log P)^{-4-2\tau/3+\varepsilon},$$
which yields the bound for $\lvert\mathcal{E}(N)\rvert$ stated at the beginning of the proposition.
\end{proof}
\emph{Proof of Theorem \ref{thm01}}. Take $n\in\mathbb{N}$ with $N/2\leq n\leq N$ and $n\notin  \mathcal{E}_{\delta}(N)\cup \mathcal{E}(N).$ Recalling the definitions before (\ref{exc}) and combining Propositions \ref{prop11} and (\ref{hWW}) we obtain
\begin{align*}
R(n)&
=\int_{\grm}h(\alpha)^{2}W(\alpha)^{2}e(-\alpha n)d\alpha+\int_{\grM}h^{*}(\alpha)^{2}W^{*}(\alpha)^{2}e(-\alpha n)d\alpha+\int_{\grM}\mathcal{F}(\alpha)e(-\alpha n)d\alpha
\\
&=\frak{S}(n)J(n)+O\big((HM)^{2}(\log N)^{-2-2\tau/9+\varepsilon}\big).
\end{align*}
Now fix a parameter $\upsilon<2\tau/9$. Then applying Lemmata \ref{lema3} and \ref{JN} we find that whenever $n$ is as described above with the additional condition $n\notin A_{\upsilon}(N)$ then one has
$$R(n)\gg (HM)^{2}(\log N)^{-2-\upsilon}.$$ Observe that by Lemma \ref{lema3} and Proposition \ref{prop1} the cardinality of the set of integers $N/2\leq n\leq N$ with $n\notin  \mathcal{E}_{\delta}(N)\cup \mathcal{E}(N)\cup A_{\upsilon}(N)$ is $O\big(N(\log N)^{-\upsilon}\big).$ Consequently, summing over dyadic intervals and observing that we can take $\upsilon$ to be as close to $2\tau/9$ as possible we obtain the desired result.

\end{document}